\documentclass[12pt]{article}
\usepackage{amsfonts, mathrsfs, enumerate, amsmath, amsthm, amssymb, indentfirst, mathtools}
\usepackage[colorlinks]{hyperref}
\hypersetup{linkcolor=blue, urlcolor=blue, citecolor=red}

\newtheorem{theorem}{Theorem}[section]
\newtheorem{cor}[theorem]{Corollary}
\newtheorem{prop}[theorem]{Proposition}
\newtheorem{lem}[theorem]{Lemma}
\newtheorem{question}[theorem]{Question}

\theoremstyle{definition}
\newtheorem{eg}[theorem]{Example}

\newcommand{\N}{\mathbb{N}}
\newcommand{\R}{\mathbb{R}}
\newcommand{\Q}{\mathbb{Q}}
\newcommand{\T}{\mathcal{T}}
\newcommand{\B}{\mathcal{B}}
\newcommand{\rt}{\mathfrak{r}}
\newcommand{\lt}{\mathfrak{l}}
\newcommand{\mt}{\mathfrak{m}}
\newcommand{\set}[2]{\{#1:#2\}}  

\newcommand{\transO}{\Omega^{\Omega}}
\newcommand{\symO}{\mathrm{Sym}(\Omega)}

\newcommand{\verteq}{\rotatebox{90}{$=$}}

\newcommand{\cf}{\operatorname{cf}}

\renewcommand{\P}{\mathcal{P}}

\begin{document}

\title{Topological transformation monoids}
\author{Z.\ Mesyan, J.\ D.\ Mitchell, and Y.\ P\'eresse}

\maketitle

\begin{abstract}
We investigate semigroup topologies on the full transformation monoid $\transO$ of an infinite set $\Omega$. We show that the standard pointwise topology is the weakest Hausdorff semigroup topology on $\transO$, show that this topology is the unique Hausdorff semigroup topology on $\transO$ that induces the pointwise topology on the group $\symO$ of all permutations of $\Omega$, and construct $|\Omega|$ distinct Hausdorff semigroup topologies on $\transO$. In the case where $\Omega$ is countable, we prove that the pointwise topology is the only Polish semigroup topology on $\transO$. We also show that every separable semigroup topology on $\transO$ is perfect, describe the compact sets in an arbitrary Hausdorff semigroup topology on $\transO$, and show that there are no locally compact perfect Hausdorff semigroup topologies on $\transO$ when $|\Omega|$ has uncountable cofinality.

\medskip


\end{abstract}

% 08A35 - General algebraic systems: Automorphisms, endomorphisms
% 20M20 - Group theory and generalizations: Semigroups of transformations
% 54H15 - General topology: Transformation groups and semigroups

\section{Introduction}

Recall that a \emph{topological group} is a group $G$ together with a topology on $G$ that makes the multiplication and inversion operations continuous. We shall refer to topologies of this sort as \emph{group topologies}. The discrete and trivial topologies are group topologies on every group, but the question of finding interesting group topologies has received a great deal of attention in the literature. We begin with a brief overview of this literature, to motivate our work in this paper.

Markov~\cite{Markoff1944aa} sparked significant interest in the subject when he asked whether there exist infinite groups with no nondiscrete Hausdorff group topologies. This question was answered in the affirmative by Shelah~\cite{Shelah1980aa}, assuming the continuum hypothesis, and, in ZFC, by Olshanskii~\cite{Olshanskij1980aa} in the same year. There are even infinite groups where every quotient of every subgroup has no nondiscrete Hausdorff group topologies~\cite{Klyachko2012aa}.

A substantial portion of the literature on topological groups has focused on \emph{Polish} groups, which arise in many areas of mathematics, particularly in descriptive set theory. These are topological groups where the topology is \emph{Polish}, that is, completely metrizable and separable. Compared to the situation explored by Markov, it is easy to find examples both of Polish groups and of topological groups that are not Polish. Specifically, since by the Cantor-Bendixson Theorem~\cite[Theorem 6.4]{Kechris}, every Polish space has cardinality either at most $\aleph_0$ or equal to $2^{\aleph_0}$, any topological group with cardinality greater than $2^{\aleph_0}$ is not Polish. Moreover, there is no nondiscrete Polish group topology on any free group. (This follows from the result of Dudley~\cite{Dudley1961aa} that every homomorphism from a complete metric group to a free group, with the discrete topology, is continuous.)  On the other end of the spectrum are Polish groups with infinitely many non-homeomorphic Polish group topologies. Perhaps the simplest such example is that of the additive group $\R$ of real numbers, which is a Polish group with respect to the standard topology. As vector spaces over the field $\Q$ of the rational numbers, $\R$ and $\R^n$ are isomorphic for every positive integer $n$, and so $\R$ and $\R^n$ are isomorphic as additive groups also. But $\R^n$ is homeomorphic to $\R^m$ if and only if $m = n$, and so there are infinitely many non-homeomorphic Polish group topologies on $\R$. In contrast, Shelah~\cite{Shelah1984aa} showed that it is consistent with the Zermelo-Fraenkel axioms of set theory, without the axiom of choice, that every Polish group has a unique Polish group topology. Of course, in the above example, the axiom of choice was needed to show that $\R$ and $\R^n$ are isomorphic. 

One of the most extensively studied topological groups is the group $\symO$ of all permutations of a set $\Omega$, which has a natural group topology, known as the \emph{pointwise} topology. This topology is Polish in the case where $\Omega$ is countable. Gaughan~\cite{Gaughan} showed that every Hausdorff group topology on $\symO$ contains the pointwise topology, and that there is no nondiscrete locally compact Hausdorff group topology on $\symO$. The latter answered problem 96 in the Scottish Book~\cite{Mauldin2015aa}, posed by Ulam. Kechris and Rosendal~\cite{KR} showed that any homomorphism from $\symO$ into a separable group is continuous, which together with Gaughan's result about the pointwise topology implies that there is a unique Polish group topology on $\symO$, when $\Omega$ is countable. Moreover, Rosendal~\cite{Rosendal2005aa} extended Gaughan's local compactness result to show that there are not even any non-trivial homomorphisms from $\symO$ into a locally compact Polish group. Gaughan's result regarding the pointwise topology was also recently extended, by Banakh, Guran, and Protasov~\cite{BGP}, to any subgroup of $\symO$ containing the permutations with finite support.

There are many further examples of groups with a unique Polish group topology, such as
the group of isometries of Minkowski spacetime (the usual framework for special
relativity)~\cite{Kallman2010aa}; see also \cite{Gartside2008aa,
Kallman1986aa}. Furthermore, there is a wealth of examples of infinite groups with no Polish group topologies~\cite{Cohen2016aa, Mann2017aa}. Additional references include \cite{Chang2017aa, Dixon, Hart2014aa}.

\emph{Topological semigroups}, that is, semigroups with topologies that make multiplication continuous, have received somewhat less attention in the literature than topological groups. However, some notable recent papers on the topic include~\cite{Banakh2014aa, Bodirsky2017aa, Bodirsky2017-2}.

In this paper we focus on the natural semigroup analogue of $\symO$, namely the semigroup $\transO$ of all functions from $\Omega$ to $\Omega$, which also has a natural \emph{pointwise} topology. Our goal is to explore the pointwise topology on $\transO$ in detail, along with semigroup topologies on $\transO$ in general. In Section~\ref{basics-section} we give a brief review of the basics of topological semigroups, along with some methods for constructing topologies on semigroups. We then show in Theorem~\ref{point-finer} that the pointwise topology is the weakest $T_1$ (and hence also Hausdorff) semigroup topology on $\transO$. More generally, the same holds for the topology induced by the pointwise topology on any subsemigroup of $\transO$ that contains all the elements of ranks $1$ and $2$. This is analogous to the result of Gaughan regarding the pointwise topology on $\symO$ mentioned above. 

Using Theorem~\ref{point-finer} we show in Proposition~\ref{group-top} that the pointwise topology is the unique $T_1$ semigroup topology on $\transO$ that induces the pointwise topology on $\symO$. From this we conclude in Theorem~\ref{polish-thrm} that, analogously to a result about $\symO$ stated above, the pointwise topology is the only Polish semigroup topology on $\transO$, when $\Omega$ is countably infinite. We also show in Theorem~\ref{many-isolated} that if $\Omega$ is infinite, then there are either no isolated points or $2^{\cf(|\Omega|)}$ isolated points in any semigroup topology on $\transO$, where $\cf(|\Omega|)$ denotes the cofinality of $|\Omega|$. In particular, every separable semigroup topology on $\transO$ is perfect (i.e., has no isolated points), when $\Omega$ is infinite (Corollary~\ref{separable-perfect}). 

In Section~\ref{compact-section} we describe the compact sets in an arbitrary $T_1$ semigroup topology on $\transO$ (Proposition~\ref{compact-prop}). We then show that  there are no locally compact perfect $T_1$ semigroup topologies on $\transO$ when $|\Omega|$ has uncountable cofinality (Theorem~\ref{loc-comp-perf}). This is a partial analogue of Gaughan's local compactness result mentioned above.

Along the way we give various examples to illustrate our results. For instance, we construct $|\Omega|$ distinct Hausdorff semigroup topologies on $\transO$, for $\Omega$ infinite (Proposition~\ref{top-chain-prop}). We also construct a natural perfect Hausdorff semigroup topology on $\transO$ for $\Omega$ of regular cardinality, which is separable when $\Omega$ is countable, and has a very different flavour from the pointwise topology (Proposition~\ref{open-prod-top}). The paper concludes with an open question.

\section{Topological semigroups} \label{basics-section}

A \emph{topological semigroup} is a semigroup $S$ together with a topology on
$S$, such that the semigroup multiplication, viewed as a function $\mt : S\times S \to S$, is continuous; where the topology on $S\times S$ is the corresponding product topology.  A semigroup is \emph{semitopological} if for every $s\in S$ the maps $\rt_s:S\rightarrow S$ and $\lt_s:S\rightarrow S$ induced by right, respectively left, multiplication by $s$ are continuous with respect to the relevant topology. It is a standard and easily verified fact that every topological semigroup is semitopological.

Next we give several simple but useful methods for constructing topologies on semigroups.

\begin{lem} \label{hom-lemma}
Let $f : S_1 \to S_2$ be a homomorphism of semigroups. Suppose that $S_2$ is a topological semigroup with respect to a topology $\, \T_2$, and let $\, \T_1$ be the least topology on $S_1$ such that $f$ is continuous. Then $\, \T_1 = \set{(A)f^{-1}}{A \in \T_2}$, and $S_1$ is a topological semigroup with respect to $\, \T_1$.
\end{lem}

\begin{proof}
Noting that $(\bigcup_{i \in I} A_i)f^{-1} = \bigcup_{i \in I} (A_i)f^{-1}$ and $(\bigcap_{i \in I} A_i)f^{-1} = \bigcap_{i \in I} (A_i)f^{-1}$ for any collection $\set{A_i}{i \in I}$ of subsets of $S_2$, it follows that $\set{(A)f^{-1}}{A \in \T_2}$ is a topology on $S_1$. Clearly, this topology is contained in $\T_1$, and $f$ is  continuous with respect to it. Therefore $\T_1 = \set{(A)f^{-1}}{A \in \T_2}$.

Next, let $\mt_i : S_i \times S_i \to S_i$ denote the multiplication map on $S_i$, for $i \in \{1, 2\}$, and define $f_c : S_1 \times S_1 \to S_2 \times S_2$ by $(x,y)f_c = ((x)f,(y)f)$ for all $x,y \in S_1$. Then, viewing $S_1 \times S_1$ and $S_2 \times S_2$ as topological spaces in the product topologies induced by $\T_1$ and $\T_2$, respectively, $f_c$ is continuous, since both of its coordinate functions, namely $f : S_1 \to S_2$, are continuous. (See, e.g.,~\cite[Exercise 18.10]{Munkres}.) 

To show that $\mt_1$ is continuous, let $A \in \T_1$, and let $A'\in \T_2$ be such that $A = (A')f^{-1}$. Then $$(x,y) \in (A)\mt_1^{-1} \Leftrightarrow xy \in (A')f^{-1} \Leftrightarrow (x)f(y)f \in A'$$ $$\Leftrightarrow ((x)f,(y)f)\in (A')\mt_2^{-1} \Leftrightarrow (x,y)\in (A')\mt_2^{-1}f_c^{-1}$$ for all $x,y \in S_1$, and hence $(A)\mt_1^{-1} = (A')\mt_2^{-1}f_c^{-1}$. Since the composite $f_c \mt_2 : S_1 \times S_1 \to S_2$ of continuous functions is continuous, and $A' \in \T_2$, it follows that $(A)\mt_1^{-1}$ is open in $S_1\times S_1$. Hence $\mt_1$ is continuous, and therefore $S_1$ is a topological semigroup with respect to $\T_1$.
\end{proof}

\begin{lem} \label{semi-top-lemma}
Let $S$ be a semigroup.
\begin{enumerate}
\item[$(1)$] Given an ideal $I$ of $S$, let $\, \T_1 = \set{A}{A \subseteq S\setminus I}\cup\{S\}$ and $\, \T_2 = \T_1 \cup \set{A \cup I}{A \in \T_1}$. Then $S$ is a topological semigroup with respect to $\, \T_1$ and $\, \T_2$.

\item[$(2)$] Let $\, \T_1$ and $\, \T_2$ be topologies on $S$, and let $\, \T_3$ be the topology generated by $\, \T_1 \cup \T_2$. If $S$ is a topological semigroup with respect to $\, \T_1$ and $\, \T_2$, then the same holds for $\, \T_3$.
\end{enumerate}
\end{lem}

\begin{proof}
Let $\mt : S \times S \to S$ denote the multiplication map on $S$.

(1) It is easy to see that $\T_1$ is closed under unions and (finite) intersections, and is hence a topology on $S$. Since $I$ is an ideal, $(A) \mt^{-1} \subseteq (S\setminus I) \times (S\setminus I)$, and hence $(A) \mt^{-1}$ is open in the product topology on $S \times S$ induced by $\T_1$, for any $A \in \T_1 \setminus \{S\}$. It follows that $S$ is a topological semigroup with respect to $\T_1$.

Next, let $f: S \to S/I$ be the natural homomorphism. Then the elements of $\T_2$ are precisely the preimages under $f$ of the subsets open in the discrete topology on $S/I$. Thus $S$ is a topological semigroup with respect to $\T_2$, by Lemma~\ref{hom-lemma}. 

(2) Let $A \in \T_3$. Then $A = \bigcup_i (B_i \cap C_i)$ for some $B_i \in \T_1$ and $C_i \in \T_2$. Thus $$(A)\mt^{-1} = \bigcup_i ((B_i \cap C_i)\mt^{-1}) = \bigcup_i ((B_i)\mt^{-1} \cap (C_i)\mt^{-1}).$$ Since $S$ is a topological semigroup with respect to $\T_1$ and $\T_2$, for each $i$ we can write $(B_i)\mt^{-1}  = \bigcup_j (B_{ij} \times B'_{ij})$ and $(C_i)\mt^{-1}  = \bigcup_l (C_{il} \times C'_{il})$, for some $B_{ij}, B'_{ij} \in \T_1$ and $C_{il}, C'_{il} \in \T_2$. Thus $$(A)\mt^{-1} = \bigcup_i \bigcup_j \bigcup_l ((B_{ij} \times B'_{ij}) \cap (C_{il} \times C'_{il})) = \bigcup_i \bigcup_j \bigcup_l ((B_{ij} \cap C_{il}) \times (B'_{ij} \cap C'_{il})),$$ since it is easy to see that $$(X_1\times Y_1) \cap (X_2 \times Y_2) = (X_1 \cap X_2) \times (Y_1\cap Y_2)$$ for any sets $X_1$, $X_2$, $Y_1$, $Y_2$. Hence $(A)\mt^{-1}$ is open in the product topology on $S\times S$ induced by $\T_3$. It follows that $S$ is a topological semigroup with respect to $\T_3$.
\end{proof}

The two obvious topologies on a semigroup $S$ can be viewed as extremal cases of the constructions in Lemma~\ref{semi-top-lemma}(1). Specifically, if $I=S$, then $\T_1 = \T_2$ is the trivial topology, while if $I = \emptyset$, then $\T_1 = \T_2$ is the discrete topology. As we shall see in Proposition~\ref{top-chain-prop}, the two topologies $\T_1$ and $\T_2$ can also be distinct from each other.

Recall that a topological space $X$ is $T_1$ if for any two distinct points $x, y \in X$ there is an open neighbourhood of $x$ that does not contain $y$, and $X$ is $T_2$, or \emph{Hausdorff}, if for any two distinct points $x,y \in X$ there are open neighbourhoods $U$ and $V$ of $x$ and $y$, respectively, such that $U \cap V = \emptyset$. 

We conclude this section with an observation that, while interesting, will not be used in the rest of the paper.

\begin{prop}\label{prop-congruence}
Let $S$ be a topological semigroup with respect to a topology $\, \T$, define $$A_x = \bigcap\set{U}{U\in \T \text{ and } x \in U}$$ for every $x\in S$, and let $$\rho_S = \set{(x, y)}{A_x = A_y} \subseteq S\times S.$$ Then the following hold.
\begin{enumerate}
\item[$(1)$] The relation $\rho_S$ is a congruence on $S$. 
\item[$(2)$] If $\, \T$ is $T_1$, then $\rho_S = \set{(x, x)}{x \in S}$. 
\item[$(3)$] The topology $\, \T$ is contained in the least topology on $S$ with respect to which the natural homomorphism $S \to S/\rho_S$ is continuous, where $S/\rho_S$ is endowed with the discrete topology. 
\end{enumerate}
\end{prop}

\begin{proof}
(1) It is clear that $\rho_S$ is an equivalence relation on $S$, and so it suffices to show that $(xs, ys), (sx, sy)\in \rho_S$ for all $s\in S$ and $(x,y) \in \rho_S$. Thus suppose that $(x,y) \in \rho_S$, and let $U$ is an open neighbourhood of $xs$, for some $s \in S$. By the continuity of multiplication, there exist open neighbourhoods $V$ and $W$ of $x$ and $s$, respectively, such that $VW \subseteq U$. But $A_x = A_y$, and so, in particular, $y \in V$. Thus $ys\in VW \subseteq U$, which shows that $ys \in A_{xs}$. By symmetry, also $xs \in A_{ys}$, giving $(xs, ys)\in \rho_S$. The proof that $(sx, sy)\in \rho_S$ is dual.

(2) If $\T$ is $T_1$, then clearly $A_x = \{x\}$ for all $x \in S$, which implies that  $\rho_S = \set{(x, x)}{x \in S}$.

(3) Let $f: S \to S/\rho_S$ be the natural homomorphism defined by $(s)f = s/\rho_S$, and let $$\P= \set{(V)f^{-1}}{V\subseteq S/\rho_S}$$ be the least topology such that $f$ is continuous (by Lemma~\ref{hom-lemma}). It is straightforward to verify that $U\in \P$ if and only if $U$ is a union of $\rho_S$-classes. Hence to prove that $\T \subseteq \P$, it suffices to show that if $x\in V\in \T$, then $y\in V$ for all $y\in S$ such that $(x, y)\in \rho_S$. But if $(x, y)\in \rho_S$, then every open neighbourhood of $x$ is also an open neighbourhood of $y$, and hence $y\in V$, as required.
\end{proof}

\section{The pointwise topology}

Given a set $\Omega$, we denote by $\transO$ the \emph{full transformation monoid} of $\Omega$, consisting of all functions from $\Omega$ to $\Omega$, under composition. The \emph{pointwise} (or \emph{function}, or \emph{finite}) \emph{topology} on $\transO$ has a base of open sets of the following form: $$[\sigma : \tau] = \set{f \in \transO}{(\sigma_i)f = \tau_i \text{ for all } 0 \leq i\leq n},$$ where $\sigma=(\sigma_0, \sigma_1, \dots, \sigma_n)$ and $\tau=(\tau_0, \tau_1, \dots, \tau_n)$ are sequences of elements of $\Omega$, and $n \in \N$ (the set of the natural numbers). It is straightforward to see that this coincides with the product topology on $\transO = \prod_\Omega \Omega$, where each component set $\Omega$ is endowed with the discrete topology.  As a product of discrete spaces, this space is Hausdorff. It is well-known and easy to see that $\transO$ is a topological semigroup with respect to the pointwise topology. If $\Omega$ is finite, then $\transO$ is discrete in this topology. Finally, it is a standard fact that if $\Omega$ is countable, then the pointwise topology on $\transO$ is Polish, i.e., separable and completely metrisable (see, e.g.,~\cite[Section 3.A, Example 3]{Kechris}).

We are now ready to prove an analogue of a result of Gaughan~\cite[Theorem 1]{Gaughan} for infinite symmetric groups, which shows that the pointwise topology is the weakest Hausdorff semigroup topology on $\transO$. 

Recall that the \emph{rank} of a function $f \in \transO$ is the cardinality of the image $(\Omega)f$ of $f$.

\begin{theorem} \label{point-finer}
Let $\, \Omega$ be a set, let $S$ be a subsemigroup of $\, \transO$ that contains all the transformations of rank at most $\, 2$, and let $\, \T$ be a topology on $S$ with respect to which $S$ is a semitopological semigroup. Then the following are equivalent.
\begin{enumerate}
\item[$(1)$] $\, \T$ is Hausdorff.
\item[$(2)$] $\, \T$ is $T_1$.
\item[$(3)$] Every set of the form $\, [\sigma : \tau] \cap S$ is open in $\, \T$. \emph{(}I.e., $\, \T$ contains the topology induced on $S$ by the pointwise topology on $\, \transO$.\emph{)}
\item[$(4)$] Every set of the form $\, [\sigma : \tau] \cap S$ is closed in $\, \T$.
\end{enumerate}
\end{theorem}

\begin{proof} 
We may assume that $\Omega$ is infinite, since each of the conditions (1)--(4) is equivalent to $\T$ being discrete in the case where $\Omega$ is finite. We note, however, that our arguments below only require that $\Omega$ contains at least $2$ elements. We shall prove that $(1) \Rightarrow (2) \Rightarrow (3) \Rightarrow (1)$ and $(3) \Rightarrow (4) \Rightarrow (2)$.

$(1) \Rightarrow (2)$ This is a tautology.

$(2) \Rightarrow (3)$ Assuming that $\T$ is $T_1$, we start by showing that for any $\alpha,\beta \in \Omega$, the set $$F_{\alpha,\beta}=\set{f\in S}{(\alpha)f\not=\beta}$$ is closed in $\T$. Let $\gamma \in \Omega \setminus\{\alpha\}$, let $f\in S$ be the constant function with image $\alpha$, and define $g\in S$ by 
$$(\delta)g = \left\{ \begin{array}{ll}
\gamma & \text{if } \, \delta=\beta\\
\alpha & \text{if } \, \delta \neq \beta
\end{array}\right..$$
Since $\T$ is $T_1$, the singleton $\{f\}$ is closed, and since $S$ is semitopological with respect to $\T$, the composite $\lt_f\circ \rt_g: S \rightarrow S$ of multiplication maps is continuous. Hence $$(f)(\lt_f\circ \rt_g)^{-1}=\set{h\in S}{fhg=f}=F_{\alpha,\beta}$$ must be closed in $\T$.
 
Let $\sigma=(\sigma_0, \sigma_1, \dots, \sigma_n)$ and $\tau=(\tau_0, \tau_1, \dots, \tau_n)$ be arbitrary sequences of elements of $\Omega$, of the same finite length. Then $$S \setminus ([\sigma : \tau]\cap S) = \bigcup_{i=0}^n F_{\sigma_i,\tau_i},$$ is closed, being a finite union of closed sets. Thus $[\sigma : \tau] \cap S$ is open in $\T$.

$(3) \Rightarrow (1)$ Since the pointwise topology on $\transO$ is Hausdorff, so is the topology it induces on $S$, and hence so is any topology on $S$ that contains this induced topology.

$(3) \Rightarrow (4)$ Let $\sigma=(\sigma_0, \sigma_1, \dots, \sigma_n)$ and $\tau=(\tau_0, \tau_1, \dots, \tau_n)$ be arbitrary finite sequences of elements of $\Omega$, of length $n+1$. Then $$S \setminus ([\sigma : \tau]\cap S) = \bigcup_{\phi \neq \tau} ([\sigma : \phi] \cap S),$$ where $\phi=(\phi_0, \phi_1, \dots, \phi_n)$ runs over all sequences of elements of $\Omega$ of length $n+1$, distinct from $\tau$. Thus, if each $[\sigma : \phi] \cap S$ is open in $\T$, then so is $S \setminus ([\sigma : \tau]\cap S)$. In this case $[\sigma : \tau] \cap S$ must be closed in $\T$.

$(4) \Rightarrow (2)$ Let $f \in S$ be any function. Then $$\{f\} = \bigcap_{\alpha \in \Omega} \{g \in S : (\alpha)g = (\alpha)f\} = \bigcap_{\alpha \in \Omega} ([(\alpha) : ((\alpha)f)]\cap S).$$ Thus, if (4) holds, then $\{f\}$ is closed in $\T$. Since $f \in S$ was arbitrary, it follows that $\T$ is $T_1$.
\end{proof}

We shall show in Proposition~\ref{top-chain-prop} that $\transO$ has many Hausdorff semigroup topologies different from the pointwise one.

Next we give an example of a semitopological subsemigroup $S$ of $\transO$ with a topology that is $T_1$ but not Hausdorff, to show the necessity of the hypothesis on $S$ in Theorem~\ref{point-finer}.

Given a set $X$ we denote by $|X|$ the cardinality of $X$.

\begin{eg}
Let $\Omega$ be an infinite set, let $S$ be an infinite subsemigroup of $\transO$ consisting of bijections (so $S$ is a subsemigroup of $\symO$), and let $\T$ be the cofinite topology on $S$. That is, $\T$ consists of precisely $\emptyset$ and the cofinite subsets of $S$ (i.e., $X \subseteq S$ such that $|S\setminus X| < \aleph_0$). Clearly, $\T$ is $T_1$. (In fact, the cofinite topology is the weakest $T_1$ topology on any set.) On the other hand, since $S$ is infinite, the intersection of any two nonempty elements of $\T$ is infinite, and hence $\T$ is not Hausdorff. We shall show that $S$ is semitopological with respect to $\T$.

Let $U \in \T \setminus \{S\}$ be a nonempty open set, write $S\setminus U = \{g_0, \dots, g_n\}$ for some $n \in \N$, and let $f \in S$. Since $f$ is a bijection, and hence invertible both on the left and the right (in $\symO$), for each $i \in \{0, \dots, n\}$ there can be at most one $h \in S$ such that $fh = g_i$, and at most one $h \in S$ such that $hf = g_i$. Hence $$(U)\lt_f^{-1} = \set{h\in S}{fh \in U} = \set{h \in S}{fh \notin \{g_0, \dots, g_n\}},$$ and therefore $|S \setminus (U)\lt_f^{-1}| \leq n+1$. Similarly, $|S \setminus (U)\rt_f^{-1}| \leq n+1$. It follows that $(U)\lt_f^{-1}, (U)\rt_f^{-1} \in \T$ for all $U \in \T$, and hence $S$ is semitopological with respect to $\T$. \hfill \qedsymbol
\end{eg}

Given a set $\Omega$, the \emph{pointwise topology} on the group $\symO$ of all permutations of $\Omega$ is the subspace topology induced on $\symO$ by the pointwise topology on $\transO$. It is well-known and easy to see that $\symO$ is a topological group with respect to the pointwise topology. (That is, $\symO$ is a topological semigroup with respect to this topology, and the inversion map $\symO \to \symO$ is continuous.) Moreover, as with $\transO$, if $\Omega$ is countable, then the pointwise topology on $\symO$ is Polish (see, e.g.,~\cite[Section 9.A, Example 7]{Kechris}).

It turns out that the pointwise topology on $\transO$ is the only $T_1$ semigroup topology that induces the pointwise topology on $\symO$.

\begin{prop} \label{group-top}
Let $\, \Omega$ be a set, and let $\, \P$ denote the pointwise topology on $\, \transO$. Suppose that $\, \transO$ is a semitopological semigroup with respect to a $T_1$ topology $\, \T$, and that $\, \T$ induces a topology on $\, \symO$ that is contained in the pointwise topology. Then $\, \T = \P$.
\end{prop}

\begin{proof}
We may assume that $\Omega$ is infinite, since otherwise the only $T_1$ topology on $\transO$ is the discrete topology, which coincides with $\P$ in this case.

Write $\Omega = \bigcup_{\alpha \in \Omega} \Sigma_{\alpha}$, where the union is disjoint, and $|\Sigma_{\alpha}| = |\Omega|$ for each $\alpha \in \Omega$. Let $g_1 \in \transO$ be an injective function such that $|\Omega \setminus (\Omega)g_1| = |\Omega|$, and let $g_2 \in \transO$ be the function that takes each $\Sigma_{\alpha}$ to $\alpha$. Then it is easy to show that $g_1\symO g_2 = \transO$ (see the proof of~\cite[Theorem 12]{ZM} for the details).

Let $U \in \T$ be a nonempty open set, and let $f \in U$. Also let $h \in \symO$ be such that $g_1hg_2 = f$. Then $h \in (U)(\lt_{g_1} \circ \rt_{g_2})^{-1} \in \T$, and hence there is an open neighbourhood $V \subseteq \symO \cap (U)(\lt_{g_1} \circ \rt_{g_2})^{-1}$ of $h$ in the topology induced on $\symO$ such that $$V \supseteq \set{g \in \symO}{(\beta_0)g=(\beta_0)h, \dots, (\beta_n)g=(\beta_n)h},$$ for some distinct $\beta_0, \dots, \beta_n \in \Omega$. Let $\Xi = (\{\beta_0, \dots, \beta_n\})g_1^{-1}$, and let $$W = \set{g \in \transO}{(\alpha)g = (\alpha)g_1hg_2 \text{ for all } \alpha \in \Xi}.$$ Then for each $g \in W$, there exists $p \in V$ such that $(\alpha)g_1p \in \Sigma_{(\alpha)g}$ for all $\alpha \in \Omega$, and hence $g = g_1pg_2 \in g_1Vg_2$. Thus $f \in W \subseteq g_1V g_2 \subseteq U$. Since $f \in U$ was arbitrary and $W \in \P$, this implies that $U \in \P$, and hence $\T \subseteq \P$. Finally, since $\T$ was assumed to be $T_1$, we conclude that $\T = \P$, by Theorem~\ref{point-finer}.
\end{proof}

Recall that a topological space is \emph{separable} if it contains a countable dense subset. Kechris and Rosendal showed in~\cite[Theorem 6.26]{KR} that the pointwise topology is the only nontrivial separable group topology on $\symO$, when $\Omega$ is countably infinite. In particular, the pointwise topology is the only Polish group topology on $\Omega$ in this situation, since any Polish topology must be Hausdorff, and hence nontrivial. Using Proposition~\ref{group-top}, we can prove the analogous statement for $\transO$.

\begin{theorem} \label{polish-thrm}
Let $\Omega$ be a countably infinite set. Then the pointwise topology is the only Polish topology on $\, \transO$ with respect to which it is a semitopological semigroup.
\end{theorem}

\begin{proof}
Let $\P$ denote the pointwise topology on $\transO$, and let $\T$ be any Polish topology on $\transO$ with respect to which it is a semitopological semigroup. Then, in particular, $\T$ is Hausdorff and so $\T\supseteq \P$, by Theorem \ref{point-finer}.

It is a standard fact (see, e.g.,~\cite[Theorem 3.11]{Kechris}) that a subspace of a Polish space is Polish if and only if it is $G_{\delta}$ (i.e., a countable intersection of open sets). Since $\symO$ is Polish in the topology induced by $\P$, it follows that $\symO$ is a $G_{\delta}$ subspace of $\transO$ with respect to $\P$. From the fact that $\T \supseteq \P$, we conclude that $\symO$ is also a $G_{\delta}$ subspace of $\transO$ with respect to $\T$, and hence $\symO$ is a Polish space in the topology induced by $\T$. 
Since $\transO$ is a semitopological semigroup with respect to $\T$, it is easy to see that the same holds for $\symO$ with respect to the topology induced by $\T$, and therefore it is a Polish semigroup in this topology.

According to a result of Montgomery~\cite[Theorem 2]{Montgomery}, if a group is a semitopological semigroup with respect to a Polish topology, then it is a topological group with respect to that topology. Hence $\symO$ is a Polish group in the topology induced by $\T$, and therefore this topology must be the pointwise topology on $\symO$, by~\cite[Theorem 6.26]{KR}. Finally, Proposition~\ref{group-top} implies that $\T=\P$.
\end{proof}

We note that~\cite[Theorem 6.26]{KR} and Proposition~\ref{group-top} imply a stronger statement than what is used in the final paragraph of the above proof. Specifically, the pointwise topology is the only $T_1$ semitopological semigroup topology on $\transO$ that induces a separable group topology on $\symO$, when $\Omega$ is countably infinite. To complement this observation and Theorem~\ref{polish-thrm}, we record the following fact.

\begin{cor} \label{separable-cor}
Let $\, \Omega$ be an uncountable set. Then $\, \transO$ does not admit a separable $T_1$ topology with respect to which it is a semitopological semigroup.
\end{cor}

\begin{proof}
Suppose that $\transO$ is a semitopological semigroup with respect to a $T_1$ topology $\T$, and that $X \subseteq \transO$ is countable and dense in $\T$. By Theorem~\ref{point-finer}, $\T$ contains the pointwise topology, and hence each subset of $\transO$ of the form $[\sigma : \tau]$ contains an element of $X$. Letting $\alpha \in \Omega$ be any element, we can find some $\beta \in \Omega \setminus (\alpha)X$, since $|(\alpha)X| \leq \aleph_0$. Hence $[(\alpha) : (\beta)] \cap X = \emptyset$, producing a contradiction.
\end{proof}

\section{Topologies with isolated points}

It is well-known (see~\cite[Lemma 1]{Malcev}) that the proper ideals of $\transO$ are precisely the subsemigroups of the form $$I_\lambda = \set{f\in \transO}{|(\Omega) f|<\lambda},$$ where $\lambda$ is a cardinal satisfying $1 \leq \lambda \leq |\Omega|$. This fact allows us to construct $|\Omega|$ distinct Hausdorff semigroup topologies on $\transO$ (containing the pointwise topology). These topologies, along with all the others considered in this section, have \emph{isolated} points, i.e., points $x$ such that $\{x\}$ is open.

\begin{prop} \label{top-chain-prop}
Let $\, \Omega$ be an infinite set, let $\, \T_P$ denote the pointwise topology on $\, \transO$, and for each nonempty proper ideal $I_\lambda$ of $\, \transO$, let $$\, \T_{I_\lambda^1} = \set{A}{A \subseteq \transO\setminus I_\lambda}\cup\{\transO\}$$ and $$\, \T_{I_\lambda^2} = \T_{I_\lambda^1} \cup \set{A \cup I_\lambda}{A \in \T_{I_\lambda^1}}.$$ Also let $\, \T_{P \cup I_\lambda^i}$ be the topology on $\, \transO$ generated by $\, \T_P$ and $\, \T_{I_\lambda^i}$ $\, (i \in \{1,2\}, \ 1 < \lambda \leq |\Omega|)$. Then the following diagram consists of Hausdorff topologies with respect to which $\, \transO$ is a topological semigroup:
$$\begin{array}{ccccccccccc}
\T_P & \subset & \cdots & \subset & \T_{P \cup I_\lambda^2} & \subset & \cdots & \subset & \T_{P \cup I_3^2} & \subset & \T_{P \cup I_2^2}\\
 \verteq &&&& \cup &&&& \cup && \cup \\
\T_P & \subset & \cdots & \subset & \T_{P \cup I_\lambda^1} & \subset & \cdots & \subset & \T_{P \cup I_3^1} & \subset & \T_{P \cup I_2^1} \\
\end{array}$$ $(1 < \lambda \leq |\Omega|)$.
Moreover, $\T_{P \cup I_2^2}$ is the discrete topology.
\end{prop}

\begin{proof}
By Lemma~\ref{semi-top-lemma}(1), $\transO$ is a topological semigroup with respect to each $\T_{I_\lambda^i}$, and hence, by Lemma~\ref{semi-top-lemma}(2), the same holds for each $\T_{P \cup I_\lambda^i}$. Moreover, the topologies $\T_{P \cup I_\lambda^i}$ are Hausdorff, since they contain the pointwise topology $\T_P$.

It is easy to see that for any $1 < \lambda \leq |\Omega|$ the topology $\T_{P \cup I_\lambda^1}$ consists of all sets of the form $A\cup B$, where $A \in \T_P$ and $B \subseteq \transO\setminus I_\lambda$, while the topology $\T_{P \cup I_\lambda^2}$ consists of all sets of the form $(A\cap I_\lambda) \cup B$, where $A \in \T_P$ and $B \subseteq \transO\setminus I_\lambda$. Thus, $\T_{P \cup I_\lambda^1} \subset \T_{P \cup I_\lambda^2}$ for each $1 < \lambda \leq |\Omega|$. Moreover, given two cardinals $1 < \lambda < \kappa \leq |\Omega|$ and $i \in \{1,2\}$, $I_\lambda \subseteq I_\kappa$ and $$(A\cap I_\kappa) \cup B = (A\cap I_\lambda) \cup (B \cup (A \cap (I_\kappa \setminus I_\lambda)))$$ imply that $\T_{P \cup I_\kappa^i} \subseteq  \T_{P \cup I_\lambda^i}$. Noting that any $f \in I_{\kappa} \setminus I_{\lambda}$ is isolated in $\T_{P \cup I_\lambda^i}$, to conclude that $\T_{P \cup I_\kappa^i} \subset  \T_{P \cup I_\lambda^i}$ it suffices to show that $f \in I_{\kappa} \setminus I_{\lambda}$ is not isolated in $\T_{P \cup I_\kappa^i}$. But this follows immediately from the fact that $A \cap I_\kappa$ is either empty or infinite, for any $A \in \T_P$ and $\kappa \geq 3$, and hence any open set in $\T_{P \cup I_\kappa^i}$ containing $f$ is infinite.

Finally, $\T_{P \cup I_2^2}$ is the discrete topology, since $I_2 \cap [(\alpha) : (\alpha)] \in \T_{P \cup I_2^2}$ is the singleton set consisting of the constant function with value $\alpha$, for any $\alpha \in \Omega$, and all non-constant elements of $\transO$ are isolated in $\T_{P \cup I_2^2}$, by definition.
\end{proof}

In addition to describing the ideals of $\transO$, Mal'cev~\cite{Malcev} classified all the congruences $\rho$ on this semigroup. Thus, one can obtain additional Hausdorff semigroup topologies on $\transO$ by putting the discrete topology on $\transO/\rho$, and considering the topology generated by the pointwise topology together with the one induced on $\transO$ via Lemma~\ref{hom-lemma} (with $f : \transO \to \transO/\rho$ taken to be the natural projection). However, the topologies obtained this way typically, though not always, coincide with ones described in the previous proposition. Since non-Rees congruences on $\transO$ tend to be complicated to describe, we shall not discuss the resulting topologies in further detail here.

All the topologies on $\transO$ constructed in Proposition~\ref{top-chain-prop} have $2^{|\Omega|}$ isolated points, since for example, all the surjective elements of $\transO$ are isolated in these topologies. Our next goal is to show that any semigroup topology on $\transO$ with isolated points also must have ``many" of them. We begin with a couple of lemmas.

\begin{lem}\label{lem-dual-kernels}
Let $\, \Omega$ be a set, and suppose that $\, \transO$ is a semitopological semigroup with respect to some topology. Also suppose that $f, g\in \transO$ have the property that there exist $b_0, \ldots, b_n \in \transO$ and injective $a_0, \ldots, a_n\in \transO$, such that $$\, (\Omega) b_0 \cup \cdots \cup (\Omega) b_n = \Omega$$ and $$b_0ga_0 = \cdots = b_nga_n = f.$$ If $f$ is isolated, then so is $g$.
\end{lem}

\begin{proof}
Suppose that $i\in \{0, \ldots, n\}$ and that $h\in \transO$ is such that $b_iha_i = f$. Then $b_iha_i = b_iga_i$, and since $a_i$ is injective, $b_ih = b_ig$. In other words, $h$ and $g$ agree on $(\Omega)b_i$. Hence if $b_iha_i = f$ for all $i \in \{0, \ldots, n\}$, then $h = g$, since  $(\Omega) b_0 \cup \cdots \cup (\Omega) b_n = \Omega$. Thus $$\bigcap_{i=0}^{n}(f) (\lt_{b_i} \circ \rt_{a_i})^{-1} = \{g\},$$ and therefore if $f$ is isolated, then so is $g$. 
\end{proof}

A \emph{kernel class} of an element $f\in \transO$ is a nonempty set of the
form  $$(\beta)f^{-1}=\set{\alpha\in \Omega}{(\alpha)f=\beta}.$$ The \emph{kernel} $\ker(f)$ of $f \in \transO$ is the collection of the kernel classes of $f$. Finally, the \emph{kernel class type} of $f$ is the collection $\set{f_\kappa}{1\leq \kappa \leq |\Omega|}$, where $f_\kappa$ is the cardinality of the set of kernel classes of $f$ of size $\kappa$.

\begin{lem}\label{lem-kernels}
Let $\, \Omega$ be a nonempty set, and suppose that $\, \transO$ is a semitopological semigroup with respect to some topology. Also let $f,g \in \transO$ be such that there is a bijection $p$ from the set of kernel classes of $g$ to the set of kernel classes of $f$, with the property that $\, |(\Sigma)p| \geq |\Sigma|$ for each kernel class $\, \Sigma$ of $g$. If $f$ is isolated, then so is $g$.

In particular, if an element of $\, \transO$ is isolated, then so is every other element of $\, \transO$ with the same kernel class type.
\end{lem}

\begin{proof}
We may assume that $|\Omega| \geq 2$, since otherwise $f=g$.

Since the kernel classes of $f$ and $g$ are in one-to-one correspondence, $|(\Omega)f| = |(\Omega)g|$, and hence we can write $(\Omega) f=\set{\beta_{\iota}}{\iota \in \Gamma}$ and $(\Omega) g= \set{\gamma_{\iota}}{\iota \in \Gamma}$, for some index set $\Gamma$. Moreover, by the hypothesis on $p$, we may choose the $\beta_{\iota}$ and $\gamma_{\iota}$ such that $|(\beta_{\iota})f^{-1}| \geq |(\gamma_{\iota})g^{-1}|$ for all $\iota \in \Gamma$. Let $b\in \transO$ be any surjection such that $\left((\beta_{\iota})f^{-1}\right)b = (\gamma_{\iota})g^{-1}$ for all $\iota \in \Gamma$. Also, let $a_0, a_1 \in \transO$ be any elements that take $\gamma_{\iota}$ to $\beta_{\iota}$ for all $\iota \in \Gamma$, and that have constant value $\delta_0, \delta_1 \in \Omega$, respectively, on $\Omega \setminus (\Omega) g$, where $\delta_0 \neq \delta_1$. Then $bga_0=bga_1=f$. We shall show that $g$ is the only element of $\transO$ with this property.

Suppose that $bha_0=f$ for some $h\in \transO$.  Then $(\alpha)bh=(\alpha)bg$ for any $\alpha\in \Omega$ satisfying $(\alpha)f\not = \delta_0$. Similarly, if $bha_1=f$ and $(\alpha)f \not = \delta_1$, then $(\alpha)bh=(\alpha)bg$. Thus if both $bha_0=f$  and  $bha_1=f$, then $bh=bg$, and hence $h=g$, since $b$ is surjective and therefore left-invertible. Thus, we have shown that $$(f) (\lt_{b} \circ \rt_{a_0})^{-1} \cap (f) (\lt_{b} \circ \rt_{a_1})^{-1} = \{g\},$$ and therefore if $f$ is isolated, then so is $g$.

The final claim is immediate.
\end{proof}

The \emph{cofinality} $\cf(\kappa)$ of a cardinal $\kappa$ is the least cardinal $\lambda$ such that $\kappa$ is the union of $\lambda$ cardinals, each smaller than $\kappa$. A cardinal $\kappa$ is \emph{regular} if $\cf(\kappa) = \kappa$. It is a standard and easily verified fact that $\cf(\kappa) \leq \kappa$ for every cardinal $\kappa$. Let us also recall that for an infinite cardinal $\kappa$, the cofinality $\cf(\kappa)$ is necessarily infinite. (Otherwise it would be the case that $\kappa = \lambda_0 + \dots + \lambda_n$ for some $n \in \N$ and infinite cardinals $\lambda_i < \kappa$. But $\sum_{i=0}^n \lambda_i = \max\{\lambda_0, \dots, \lambda_n\},$ since the $\lambda_i$ are infinite, contradicting $\lambda_i < \kappa$.)

We are now ready for the main result of this section.

\begin{theorem} \label{many-isolated}
Let $\, \Omega$ be an infinite set, let $\kappa = \cf(|\Omega|)$, and suppose that $\, \transO$ is a semitopological semigroup with respect to some topology. Then there are either no isolated points or at least $\, 2^{\kappa}$ isolated points in $\, \transO$. 

In particular, if $\, |\Omega|$ is regular, then there are either no isolated points or $\, 2^{|\Omega|}$ isolated points in $\, \transO$.
\end{theorem}

\begin{proof}
Suppose that $f\in \transO$ is isolated. If $f$ has at least $\kappa$ distinct kernel classes, then the set of those $g \in \transO$ with the same kernel class type as $f$ has cardinality at least $2^{\kappa}$. (Since $\kappa$ is infinite, one can obtain $2^{\kappa}$ such $g$ by permuting $\kappa$ of the elements in the image of $f$.) Hence, there are at least $2^{\kappa}$ isolated points in $\transO$, by Lemma~\ref{lem-kernels}. Let us therefore assume that $f$ has strictly fewer than $\kappa$ distinct kernel classes, and so, in particular, $\Omega \setminus (\Omega)f \neq \emptyset$, since $\kappa \leq |\Omega|$. 

It cannot be the case that every kernel class of $f$ has cardinality strictly less than $|\Omega|$, since then the union of these (fewer than $\kappa$) kernel classes would have cardinality strictly less than $|\Omega|$, by the definition of ``cofinality". This would contradict the fact that the union of the kernel classes of any element of $\transO$ is $\Omega$. Therefore, there must be a kernel class $\Sigma$ of $f$ such that $|\Sigma| = |\Omega|$. Let $(\Sigma)f  = \beta$, and partition $\Sigma$ as $\Sigma = \Lambda_0 \cup \Lambda_1$, such that $\Lambda_0$ and $\Lambda_1$ are nonempty. Let $g\in \transO$ be defined by 
$$(\alpha)g = \left\{ \begin{array}{cl}
(\alpha)f & \text{if } \, \alpha \in \Omega\setminus \Sigma\\
\beta & \text{if } \, \alpha \in \Lambda_0\\
\gamma & \text{if } \, \alpha \in \Lambda_1
\end{array}\right.,$$
where $\gamma$ is any value in $\Omega \setminus (\Omega) f$ (which exists by assumption). Let $b_0, b_1\in \transO$ be any functions that act as the identity on $\Omega\setminus \Sigma$, and where $(\Sigma)b_0 = \Lambda_0$ and $(\Sigma)b_1 = \Lambda_1$. Also let $a_0 \in \transO$ be the identity function, and let $a_1 \in \transO$ be the transposition interchanging $\beta$ and $\gamma$. Then $$(\Omega) b_0 \cup(\Omega) b_1 = \Lambda_0 \cup \Lambda_1 \cup (\Omega\setminus \Sigma) = \Omega,$$ and $f = b_0ga_0 = b_1ga_1$. Hence $g$ is isolated, by Lemma~\ref{lem-dual-kernels}. Since there are $2^{|\Sigma|} = 2^{|\Omega|}$ ways to partition $\Sigma$ into $\Lambda_0$ and $\Lambda_1$, resulting in $2^{|\Omega|} \geq 2^{\kappa}$ functions $g$, the desired conclusion follows.

For the final claim, we recall the fact that $2^{\lambda} = \mu^{\lambda}$ for any infinite cardinal $\lambda$ and any cardinal $\mu$ satisfying $2 \leq \mu \leq \lambda$ (see, e.g.,~\cite[Lemma 5.6]{Jech}). Thus $2^{|\Omega|} = |\Omega|^{|\Omega|} = |\transO|$, since $\Omega$ is assumed to be infinite. Therefore, if $|\Omega|$ is regular and there are isolated points in $\transO$, then there are precisely $2^{\cf(|\Omega|)} = 2^{|\Omega|}$ of them.
\end{proof}

Recall that a topology is \emph{perfect} if it has no isolated points.

\begin{cor} \label{separable-perfect}
Let $\, \Omega$ be an infinite set. Then every separable topology on $\, \transO$, with respect to which $\, \transO$ is a semitopological semigroup, is perfect.
\end{cor}

\begin{proof}
Suppose that $\transO$ is a semitopological semigroup with respect to a separable topology $\T$. By Theorem~\ref{many-isolated}, if $\T$ has isolated points, then it must have at least $2^{\cf(|\Omega|)} \geq 2^{\aleph_0}$ of them, since as noted above, $\cf(|\Omega|) \geq |\Omega| \geq \aleph_0$. But since $\T$ is separable, it can have at most countably many isolated points, and therefore $\T$ must be perfect. 
\end{proof}

We conclude this section with another result about isolated points. It sheds additional light on the topologies constructed in Proposition~\ref{top-chain-prop} and strengthens Lemma~\ref{lem-kernels}.

\begin{prop}
Let $\, \Omega$ be an infinite set, and suppose that $\, \transO$ is a semitopological semigroup with respect to some topology. Let $f,g \in \transO$ be such that $f$ is isolated. If either of the following conditions holds, then $g$ is also isolated.
\begin{enumerate}
\item[$(1)$] $|(\Omega)f| \leq |(\Omega)g| < \aleph_0$.
\item[$(2)$] $|(\Omega)f| = |(\Omega)g| \geq \aleph_0$, and there is an injection $p$ from the set $K_g$ of kernel classes of $g$ to the set $K_f$ of kernel classes of $f$, with the property that $\, |(\Sigma)p| \geq |\Sigma|$ for each $\Sigma \in K_g$.
\end{enumerate}
\end{prop}

\begin{proof}
Suppose that (1) holds. Let $\Xi_0, \dots, \Xi_{n}$ be the kernel classes of $f$, where we may assume that $|\Xi_0| = |\Omega|$, and let $\Upsilon_0, \dots, \Upsilon_m$ be the kernel classes of $g$. Also, write $(\Xi_i)f = \alpha_i$ and $(\Upsilon_j)g = \beta_j$ for all $i \in \{0, \dots, n\}$ and $j \in \{0, \dots, m\}$. For each $j \in \{0, \dots, m\}$ let $a_j \in \transO$ be an injection such that $(\beta_j)a_j = \alpha_0$ and $$\{\alpha_0, \dots, \alpha_n\} \subseteq (\{\beta_0, \dots, \beta_m\})a_j.$$ Such functions exist, since $n \leq m$. Also for each $j \in \{0, \dots, m\}$ let $b_j \in \transO$ be such that $(\Xi_0)b_j = \Upsilon_j$, and $(\Xi_i)b_j \subseteq \Upsilon_l$, where $\beta_l = (\alpha_i)a_j^{-1}$, for each $i \in \{1, \dots, n\}$. Such functions exist, since $|\Xi_0| = |\Omega|$. Then $b_jga_j = f$ for each $j \in \{0, \dots, m\}$, and $$\Omega \supseteq \bigcup_{i=0}^m(\Omega)b_j \supseteq  \bigcup_{j=0}^m \Upsilon_j = \Omega,$$ which implies that $$(\Omega) b_0 \cup \cdots \cup (\Omega) b_m = \Omega.$$ Therefore $g$ is isolated, by Lemma~\ref{lem-dual-kernels}.

Now suppose that (2) holds. Upon replacing $f$ with another function having the same kernel, we may assume, by Lemma~\ref{lem-kernels}, that $|\Omega \setminus (\Omega)f| = |\Omega|$. Since the image of $g$ is infinite, we can write $$K_g = \set{\Xi_{\iota}}{\iota \in \Gamma} \cup \set{\Upsilon_{\zeta}}{\zeta \in \Delta},$$ where the union is disjoint, and $|\Gamma| = |\Delta| = |(\Omega)g|$. Let $b_0 \in \transO$ be such that $((\Xi_{\iota})p)b_0 = \Xi_{\iota}$ for each $\iota \in \Gamma$, and $b_0$ takes $K_f \setminus (\set{\Xi_{\iota}}{\iota \in \Gamma})p$ injectively into $\set{\Upsilon_{\zeta}}{\zeta \in \Delta}$. (That is, $b_0$ maps all the points in an element of $K_f \setminus (\set{\Xi_{\iota}}{\iota \in \Gamma})p$ to some $\Upsilon_{\zeta}$.) Such a transformation exists, since $|(\Xi_{\iota})p| \geq |\Xi_{\iota}|$ for each $\Xi_{\iota}$, and $|\Delta| = |(\Omega)f|$. Similarly, let $b_1 \in \transO$ be such that $((\Upsilon_{\zeta})p)b_1 = \Upsilon_{\zeta}$ for each $\zeta \in \Delta$, and $b_1$ takes $K_f \setminus (\set{\Upsilon_{\zeta}}{\zeta \in \Delta})p$ injectively into $\set{\Xi_{\iota}}{\iota \in \Gamma}$. Then $$\Omega \supseteq (\Omega) b_0 \cup (\Omega) b_1 \supseteq \bigcup_{\iota \in \Gamma} \Xi_{\iota} \cup \bigcup_{j\in \Delta} \Upsilon_{\zeta} = \Omega,$$ giving $(\Omega)b_0 \cup (\Omega)b_1 = \Omega$. Moreover, $$\ker(b_0g) = \ker(b_1g) = \ker(f),$$ and since $|\Omega \setminus (\Omega)f| = |\Omega|$, there exist injective $a_0, a_1\in \transO$ such that $b_0ga_0 = b_1ga_1 = f$. Hence, by Lemma~\ref{lem-dual-kernels}, $g$ is isolated.
\end{proof}

\section{Topologies obtained by restricting images}

Given a set $\Omega$ and collection $\set{\Sigma_{\alpha}}{\alpha \in \Omega}$ of subsets of $\Omega$, we identify $\prod_{\alpha \in \Omega} \Sigma_{\alpha}$ with $$\set{f \in \transO}{(\alpha)f \in \Sigma_{\alpha} \text{ for all } \alpha \in \Omega}.$$ Next we construct a perfect Hausdorff topology on $\transO$, different from the pointwise topology, by declaring certain sets of the form $\prod_{\alpha \in \Omega} \Sigma_{\alpha}$ open.

\begin{prop} \label{open-prod-top}
Let $\, \Omega$ be an infinite set, let $$\B = \bigg\{\prod_{\alpha \in \Omega} \Sigma_{\alpha} : |\set{\alpha \in \Omega}{\beta \notin \Sigma_{\alpha}}| <|\Omega| \text{ for all } \beta \in \Omega\bigg\},$$ and let $\, \T = \langle \B \rangle$ denote the topology on $\, \transO$ generated by $\B$. Then the following hold.
\begin{enumerate}
\item[$(1)$] $\B$ is closed under finite intersections, and in particular is a base for $\, \T$.
\item[$(2)$] Any base for $\, \T$ must have strictly more than $\, |\Omega|$ elements.
\item[$(3)$] $\T$ strictly contains the pointwise topology. In particular, $\, \T$ is Hausdorff.
\item[$(4)$] $\T$ is perfect.
\item[$(5)$] There is a subset of $\, \transO$ of cardinality $\, | \sum_{\kappa < |\Omega|} |\Omega|^{\kappa} |$ that is dense in $\, \T$. In particular, if $\, \Omega$ is countable, then $\, \T$ is separable.
\item[$(6)$] No nonempty element of $\B$ is contained in a compact subset of $\, \transO$.
\item[$(7)$] If $\, |\Omega|$ is regular, then $\, \transO$ is a topological semigroup with respect to $\, \T$. 
\end{enumerate}
\end{prop}

\begin{proof}
(1) Let $n \in \N$, and let $\prod_{\alpha \in \Omega} \Sigma_{i, \alpha} \in \B$ for $i \in \{0, \dots, n\}$. Then $$\bigcap_{i=0}^n \bigg(\prod_{\alpha \in \Omega} \Sigma_{i, \alpha}\bigg) = \prod_{\alpha \in \Omega} \bigg(\bigcap_{i=0}^n \Sigma_{i, \alpha}\bigg).$$ Now $$\bigg|\bigg\{\alpha \in \Omega : \beta \notin \bigcap_{i=0}^n\Sigma_{i, \alpha}\bigg\}\bigg| \leq \sum_{i=0}^n |\set{\alpha \in \Omega}{\beta \notin \Sigma_{i, \alpha}}| < \sum_{i=0}^n |\Omega| = |\Omega|$$ for any $\beta \in \Omega$. Hence $\bigcap_{i=0}^n (\prod_{\alpha \in \Omega} \Sigma_{i, \alpha}) \in \B$, showing that $\B$ is closed under finite intersections. Since the elements of $\B$ cover $\transO$ (e.g., because $\transO \in \B$), it follows that $\B$ is a base for $\T = \langle \B \rangle$.

(2) Let $\set{X_{\alpha}}{\alpha \in \Omega} \subseteq \T$ be a collection of nonempty open subsets of $\transO$. We shall construct sets $\Sigma_{\alpha} \subseteq \Omega$ such that $\prod_{\alpha \in \Omega}\Sigma_{\alpha} \in \B$ and $X_{\beta} \not\subseteq \prod_{\alpha \in \Omega}\Sigma_{\alpha}$ for all $\beta \in \Omega$, which implies that $\set{X_{\alpha}}{\alpha \in \Omega}$ cannot be a base for $\T$, and hence that any base for $\T$ must have strictly more than $|\Omega|$ elements.

Without loss of generality, we may assume that $\Omega$ is a cardinal. Since, by (1), $\B$ is a base for $\T$, for each $\beta \in \Omega$ we can find a nonempty $\prod_{\alpha \in \Omega}\Sigma_{\alpha, \beta} \in \B$ such that $\prod_{\alpha \in \Omega}\Sigma_{\alpha, \beta} \subseteq X_{\beta}$. Next we shall construct recursively a function $f : \Omega \to \Omega$. Assuming that $(\alpha)f$ is defined for all $\alpha < \beta \in \Omega$, let $(\beta)f$ be such that $(\beta)f \neq (\alpha)f$ for all $\alpha < \beta$, and such that $\beta \in \Sigma_{(\beta)f, \beta}$. Such $(\beta)f \in \Omega$ exists, since $$|\set{(\alpha)f}{\alpha < \beta} \cup \set{\alpha \in \Omega}{\beta \notin \Sigma_{\alpha, \beta}}| < \Omega.$$ Now for each $\alpha \in \Omega$ let $$\Sigma_{\alpha} = \left\{ \begin{array}{cl}
\Omega \setminus \{\beta\} & \text{if } \, \alpha = (\beta)f\\
\Omega & \text{if } \, \alpha \notin (\Omega)f
\end{array}\right..$$
This is well-defined, since $f$ is injective, and clearly $\prod_{\alpha \in \Omega}\Sigma_{\alpha} \in \B$. Moreover, $\prod_{\alpha \in \Omega}\Sigma_{\alpha, \beta} \not\subseteq \prod_{\alpha \in \Omega}\Sigma_{\alpha}$ for all $\beta \in \Omega$, since $\beta \in \Sigma_{(\beta)f, \beta} \setminus \Sigma_{(\beta)f}$. Hence $X_{\beta} \not\subseteq \prod_{\alpha \in \Omega}\Sigma_{\alpha}$ for all $\beta \in \Omega$, as desired.

(3) Let $\sigma=(\sigma_0, \sigma_1, \dots, \sigma_n)$ and $\tau=(\tau_0, \tau_1, \dots, \tau_n)$ be two finite sequences of elements of $\Omega$ of the same length, and let
$$\Sigma_{\alpha} = \left\{ \begin{array}{cl}
\{\tau_i\} & \text{if } \, \alpha = \sigma_i \text{ for some } i \in \{0, \dots, n\} \\
\Omega & \text{if } \, \alpha \in \Omega \setminus \{\sigma_0, \dots, \sigma_n\}
\end{array}\right..$$
Then $[\sigma : \tau] = \prod_{\alpha \in \Omega} \Sigma_{\alpha} \in \T$, showing that $\T$ contains the pointwise topology, and is hence Hausdorff. Since the sets of the form $[\sigma : \tau]$ constitute a base for the pointwise topology of cardinality $|\Omega|$, we conclude from (2) that $\T$ strictly contains the pointwise topology.

(4) If $\prod_{\alpha \in \Omega} \Sigma_{\alpha} \in \B$, then for any $\beta, \gamma \in \Omega$ there must be infinitely many $\alpha \in \Omega$ such that $\beta, \gamma \in \Sigma_{\alpha}$. Thus each nonempty element of $\B$ is (uncountably) infinite, and hence by (1), $\T$ has a base consisting of infinite sets. In particular, no point can be isolated in $\T$.

(5) Fix $\alpha \in \Omega$. Given $\Gamma \subseteq \Omega$, where $0 < |\Gamma| < |\Omega|$, let $X_{\Gamma} \subseteq \transO$ consist of all $f \in \transO$ such $(\Omega \setminus \Gamma)f = \alpha$. Then letting $X$ be the union of the $X_{\Gamma}$, we see that $X$ is dense in $\T$. Moreover $$|X| = \sum_{\Gamma} |X_{\Gamma}| = \sum_{\Gamma} |\Omega^{\Gamma}| = \sum_{\kappa < |\Omega|} |\Omega|^{\kappa}.$$ In the case where $\Omega$ is countable, $|X| = \aleph_0$, which implies that $\T$ is separable.

(6) Let $\prod_{\alpha \in \Omega}\Sigma_{\alpha} \in \B$ be nonempty, and suppose that $\prod_{\alpha \in \Omega}\Sigma_{\alpha} \subseteq X$ for some compact $X \subseteq \transO$. Then, by the definition of $\B$, we can find an infinite $\Gamma \subseteq \Omega$ and distinct $\gamma_{\beta} \in \Omega$, such that $\gamma_{\beta} \in \Sigma_{\beta}$ and $|\Sigma_{\beta}| \geq 2$ for all $\beta \in \Gamma$. For each $\beta \in \Gamma$ let $$\Xi_{\alpha, \beta}
= \left\{ \begin{array}{cl}
\{\gamma_{\beta}\} & \text{if } \, \alpha = \beta \\
\Omega & \text{if } \, \alpha \neq \beta
\end{array}\right.,$$ and let $$\Delta_{\beta}
= \left\{ \begin{array}{cl}
\Omega \setminus \{\gamma_{\beta}\} & \text{if } \, \beta \in \Gamma \\
\Omega & \text{if } \, \beta \notin \Gamma
\end{array}\right..$$
Also let $$Y = \bigg\{\prod_{\alpha \in \Omega} \Delta_{\alpha}\bigg\} \cup \bigg\{\prod_{\alpha \in \Omega} \Xi_{\alpha, \beta} : \beta \in \Gamma\bigg\}.$$ Then $\prod_{\alpha \in \Omega} \Xi_{\alpha, \beta} \in \B$ for each $\beta \in \Gamma$, and $\prod_{\alpha \in \Omega} \Delta_{\alpha} \in \B$. Moreover, $Y$ is an open cover of $\transO$, and hence also of $X$, since for each $f \in \transO \setminus \prod_{\alpha \in \Omega} \Delta_{\alpha}$ there is some $\beta \in \Gamma$ such that $(\beta)f = \gamma_{\beta}$, and hence $f \in \prod_{\alpha \in \Omega} \Xi_{\alpha, \beta}$. However, it is easy to see that $\prod_{\alpha \in \Omega}\Sigma_{\alpha}$ is not contained in the union of any finite collection of elements of $Y$, and hence neither is $X$, contradicting $X$ being compact. Therefore $\prod_{\alpha \in \Omega}\Sigma_{\alpha}$ is not contained in any compact subset of $\transO$. 

(7) Without loss of generality we may assume that $\Omega$ is a (regular) cardinal, and so in particular, it is well-ordered. Let $\prod_{\alpha \in \Omega}\Sigma_{\alpha} \in \B$ and  $f,g \in \transO$ be such that $fg \in \prod_{\alpha \in \Omega}\Sigma_{\alpha}$. We shall construct $\prod_{\alpha \in \Omega}\Gamma_{\alpha}, \prod_{\alpha \in \Omega}\Delta_{\alpha} \in \B$ such that $f \in \prod_{\alpha \in \Omega}\Gamma_{\alpha}$, $g \in \prod_{\alpha \in \Omega}\Delta_{\alpha}$, and 
\begin{equation*}
\tag{$\dagger$} \bigg(\prod_{\alpha \in \Omega}\Gamma_{\alpha}\bigg)\bigg(\prod_{\alpha \in \Omega}\Delta_{\alpha}\bigg) \subseteq \prod_{\alpha \in \Omega}\Sigma_{\alpha}.
\end{equation*}
The existence of such sets implies that the multiplication map on $\transO$ is continuous with respect to $\T$, and hence that $\transO$ is a topological semigroup.

For each $\alpha \in \Omega$ let $$\Gamma_{\alpha} = \{(\alpha)f\} \cup \set{\beta \in \Omega}{(\beta)g \in \Sigma_{\alpha} \text{ and } \beta < \alpha}$$ and $$\Delta_{\alpha} = \bigcap_{\set{\beta \in \Omega}{\alpha \in \Gamma_{\beta}}} \Sigma_{\beta}.$$ Then clearly $f \in \prod_{\alpha \in \Omega}\Gamma_{\alpha}$. If $\beta \notin \Gamma_{\alpha}$ for some $\alpha, \beta \in \Omega$, then either $\alpha \leq \beta$ or $(\beta)g \notin \Sigma_{\alpha}$. Since $$|\set{\alpha \in \Omega}{\alpha \leq \beta} \cup \set{\alpha \in \Omega}{(\beta)g \notin \Sigma_{\alpha}}| <\Omega$$ for all $\beta \in \Omega$, we see that $|\set{\alpha \in \Omega}{\beta \notin \Gamma_{\alpha}}| <\Omega$ for all $\beta \in \Omega$, and hence $\prod_{\alpha \in \Omega}\Gamma_{\alpha} \in \B$. 

To show that $g \in \prod_{\alpha \in \Omega}\Delta_{\alpha}$, let $\alpha \in \Omega$, and let $\beta \in \Omega$ be such that $\alpha \in \Gamma_{\beta}$. Then either $\alpha = (\beta)f$, or $(\alpha)g \in \Sigma_{\beta}$. In both cases, $(\alpha)g \in \Sigma_{\beta}$, and hence $(\alpha)g \in \Delta_{\alpha}$, as desired. 

To prove that $\prod_{\alpha \in \Omega}\Delta_{\alpha} \in \B$ we note that for any $\beta \in \Omega$, $$|\set{\alpha \in \Omega}{\beta \notin \Delta_{\alpha}}| = |\set{\alpha \in \Omega}{\exists \gamma \in \Omega \ (\alpha \in \Gamma_{\gamma} \land \beta \notin \Sigma_{\gamma})}| = \Big|\bigcup_{\set{\gamma \in \Omega}{\beta \notin \Sigma_{\gamma}}} \Gamma_{\gamma}\Big| < \Omega,$$ since $|\set{\gamma \in \Omega}{\beta \notin \Sigma_{\gamma}}| < \Omega$,  $|\Gamma_{\gamma}| < \Omega$ for each $\gamma \in \Omega$, and $\Omega$ is regular.

Finally, let $\gamma \in \Omega$, $f' \in \prod_{\alpha \in \Omega}\Gamma_{\alpha}$, and $g' \in \prod_{\alpha \in \Omega}\Delta_{\alpha}$. Then $(\gamma)f' \in \Gamma_{\gamma}$, and so $$(\gamma)f'g' \in \Delta_{(\gamma)f'} = \bigcap_{\set{\beta \in \Omega}{(\gamma)f' \in \Gamma_{\beta}}} \Sigma_{\beta} \subseteq \Sigma_{\gamma}.$$ Hence $f'g' \in \prod_{\alpha \in \Omega}\Sigma_{\alpha}$, which shows that $(\dagger)$ holds.
\end{proof}

The next result suggests that making sets of the form $\prod_{\alpha \in \Omega} \Sigma_{\alpha}$ open, other than those in the above construction, tends to result in semigroup topologies on $\transO$ that are not perfect. More specifically, only non-perfect semigroup topologies on $\transO$ have open sets $\prod_{\alpha \in \Omega} \Sigma_{\alpha}$ not of the above form that contain constant functions. In particular, no perfect semigroup topology on $\transO$ can have an open set of the form $\prod_{\alpha \in \Omega} \Sigma$, where $\Sigma$ is a proper subset of $\Omega$.

\begin{prop} \label{iso-prop}
Let $\, \Omega$ be an infinite set, and suppose that $\, \transO$ is a semitopological semigroup with respect to some topology. 
\begin{enumerate}
\item[$(1)$] Suppose that there exist $\beta \in \Omega$, a constant function $f \in \transO$, and an open neighbourhood $U$ of $f$, such that $U \subseteq \prod_{\alpha \in \Omega} \Sigma_{\alpha}$ and $\, |\set{\alpha \in \Omega}{\beta \notin \Sigma_{\alpha}}| = |\Omega|$. Then $f$ is isolated.

\item[$(2)$] Let $\, \set{\Sigma_{\alpha}}{\alpha \in \Omega}$ be a collection of disjoint nonempty subsets of $\, \Omega$, and suppose that $U \subseteq \prod_{\alpha \in \Omega} \Sigma_{\alpha}$ is open. Then every element of $U$ is isolated.
\end{enumerate}
\end{prop}

\begin{proof}
(1) Write $(\Omega)f = \gamma$, set $\Gamma = \set{\alpha \in \Omega}{\beta \in \Sigma_{\alpha}}$, and let $g \in \transO$ be defined by $$(\alpha)g 
= \left\{ \begin{array}{cl}
\beta & \text{if } \, \alpha \in \Omega\setminus \{\gamma\}\\
\gamma & \text{if } \, \alpha = \gamma
\end{array}\right..$$ Then $$(U)\rt_g^{-1} \subseteq \bigg(\prod_{\alpha \in \Omega} \Sigma_{\alpha}\bigg)\rt_g^{-1} \subseteq \set{y \in \transO}{(\alpha)y = \gamma \text{ for all } \alpha \in \Omega \setminus \Gamma},$$ and $f \in (U)\rt_g^{-1}$. If $\Gamma = \emptyset$, then this implies that $(U)\rt_g^{-1} = \{f\}$, and this set is open, since $U$ is. Hence we may assume that $\Gamma \neq \emptyset$. Since $|\Omega \setminus \Gamma| = |\Omega|$, there exists an $h \in \transO$ such that $(\Omega \setminus \Gamma)h = \Gamma$. Then $$(U)\rt_g^{-1} \circ \lt_h^{-1} \subseteq \bigg(\prod_{\alpha \in \Omega} \Sigma_{\alpha}\bigg)\rt_g^{-1} \circ \lt_h^{-1} \subseteq \set{y \in \transO}{(\alpha)y = \gamma \text{ for all } \alpha \in \Gamma},$$ and $f \in (U)\rt_g^{-1} \circ \lt_h^{-1}$. Since $U$ is open, it follows that $(U)\rt_g^{-1} \cap ((U)\rt_g^{-1} \circ \lt_h^{-1}) = \{f\}$ is open.

(2) We may assume that $U\neq \emptyset$, since otherwise there is nothing to prove. Let $f \in U$ be arbitrary. For each $\alpha \in \Omega$ we wish to construct $\Xi_{\alpha} \subseteq \Omega$, such that the following properties are satisfied:
\begin{enumerate}
\item[$($a$)$] $|\Sigma_{\alpha}|\leq |\Xi_{\alpha}|$ for all $\alpha \in \Omega$, 
\item[$($b$)$] $\Sigma_{\alpha} \cap \Xi_{\alpha} = \{(\alpha)f\}$ for all $\alpha \in \Omega$, 
\item[$($c$)$] $\Xi_{\alpha} \cap \Xi_{\beta} = \emptyset$ for all distinct $\alpha, \beta \in \Omega$,
\item[$($d$)$] $\bigcup_{\alpha \in \Omega} \Xi_{\alpha} = \Omega$.
\end{enumerate}
First, partition $\Omega$ as $\Omega = \bigcup_{\alpha \in \Omega} \Lambda_{\alpha}$, where $|\Lambda_{\alpha}| = |\Omega|$ for each $\alpha \in \Omega$. Then $|(\bigcup_{\beta \in \Lambda_{\alpha}} \Sigma_{\beta})\setminus \Sigma_{\alpha}| = |\Omega|$ for each $\alpha \in \Omega$, since the $\Sigma_{\alpha}$ are disjoint and $|\Lambda_{\alpha}| = |\Omega|$. Hence for each $\alpha \in \Omega$ we can construct $\Xi_{\alpha}$ that satisfies (a) and (b) by choosing $|\Sigma_{\alpha}|$ elements from $(\bigcup_{\beta \in \Lambda_{\alpha}} \Sigma_{\beta})\setminus \Sigma_{\alpha}$, along with $(\alpha)f$. Since the $\Lambda_{\alpha}$ are disjoint, the sets so constructed also satisfy (c). Now add any remaining elements of $\Omega$ to the $\Xi_{\alpha}$ in any way that preserves $\Sigma_{\alpha} \cap \Xi_{\alpha} = \{(\alpha)f\}$ for all $\alpha \in \Omega$. The resulting sets $\Xi_{\alpha}$ will then satisfy (a), (b), (c), and (d).

Since the $\Xi_{\alpha}$ satisfy (a), (b), and (c), we can find a $g\in \transO$ such that $(\Xi_{\alpha})g = \Sigma_{\alpha}$ and $(\alpha)fg = (\alpha)f$ for all $\alpha \in \Omega$. If $h\in \transO$ is such that $hg \in U$, then $(\alpha)hg \in \Sigma_{\alpha}$, and hence $(\alpha)h\in (\Sigma_{\alpha})g^ {-1} = \Xi_{\alpha}$, for all $\alpha \in \Omega$ (since the $\Sigma_{\alpha}$ are disjoint and $\bigcup_{\alpha \in \Omega} \Xi_{\alpha} = \Omega$).  Thus $(U)\rt_g^{-1} \subseteq \prod_{\alpha \in \Omega}\Xi_{\alpha}$, and therefore $U \cap (U)\rt_g^{-1} = \{f\}$, in light of condition (b). Since $U$ is open, it follows that $f$ is isolated.
\end{proof}

\section{Compactness} \label{compact-section}

Our next goal is to describe the compact sets in an arbitrary $T_1$ semigroup topology on $\transO$. We begin with a characterisation of the compact sets in the pointwise topology on $\transO$. If $\Omega$ is countable, in which case the pointwise topology on $\transO$ is Polish, and hence metrisable, this characterisation can be obtained from the standard fact that a subset of a metric space is compact if and only if it is complete and totally bounded. Proving the fact in question for arbitrary $\Omega$ is also easy, but we provide the details for the convenience of the reader.

Given a topological space $X$, a subset $U$ of $X$ is \emph{nowhere dense} if $X \setminus \overline{U}$ is dense in $X$, where $\overline{U}$ denotes the closure of $U$.

\begin{lem}\label{compact-lem}
Let $\, \Omega$ be a set, let $X \subseteq \transO$, and let $\, \T$ denote the pointwise topology on $\, \transO$. Then the following hold.
\begin{enumerate}
\item[$(1)$] $X$ is compact in $\, \T$ if and only if $X$ is closed in $\, \T$ and $\, |(\alpha)X| < \aleph_0$ for all $\alpha \in \Omega$.
\item[$(2)$] If $\, \Omega$ is infinite and $X$ is compact in $\, \T$, then $X$ is nowhere dense.
\end{enumerate}
\end{lem}

\begin{proof}
(1) We may assume that $\Omega$ is nonempty, since otherwise every subset of $\transO$ is both compact and closed in $\T$.

Suppose that $X$ is compact. It is a standard fact that in a Hausdorff space every compact subset is closed (see, e.g.,~\cite[Theorem 26.3]{Munkres}), and hence $X$ is closed in $\T$.   Now, let $\alpha \in \Omega$. Then $$X = \bigcup_{\beta \in \Omega} ([(\alpha) : (\beta)] \cap X).$$ Since $X$ compact in $\T$, $$X = \bigcup_{i=0}^n ([(\alpha) : (\beta_i)] \cap X)$$ for some $\beta_0, \dots, \beta_n \in \Omega$ and $n \in \N$. Hence $(\alpha)X \subseteq \{\beta_0, \dots, \beta_n\}$, giving $|(\alpha)X| < \aleph_0$.

Conversely, suppose that $X$ is closed in $\T$ and $|(\alpha)X| < \aleph_0$ for all $\alpha \in \Omega$. For each $\alpha \in \Omega$ let $\Sigma_{\alpha} = (\alpha)X$, and let $Y = \prod_{\alpha \in \Omega} \Sigma_{\alpha} \subseteq \transO$ consist of all functions $f \in \transO$ such that $(\alpha)f \in \Sigma_{\alpha}$. Then $X \subseteq Y$.  Since every closed subset of a compact set is compact (see, e.g.,~\cite[Theorem 26.2]{Munkres}), it suffices to show that $Y$ is compact. Now the subspace topology on $Y$ induced by $\T$ is precisely the product topology on $\prod_{\alpha \in \Omega} \Sigma_{\alpha}$ resulting from endowing each $\Sigma_{\alpha}$ with the discrete topology. Since each $\Sigma_{\alpha}$ is finite, and hence compact, $Y$ is a product of compact sets. Therefore $Y$ compact, by Tychonoff's theorem (see, e.g.,~\cite[Theorem 37.3]{Munkres}), as desired.

(2) If $X$ is compact, then by (1), for each $\alpha \in \Omega$ the set $(\alpha)X = \Sigma_{\alpha}$ is finite. As before we have $X \subseteq Y = \prod_{\alpha \in \Omega} \Sigma_{\alpha}$. Now let $\sigma$ and $\tau$ be any two sequences of elements of $\Omega$ of the same finite length. Since $\Omega$ is infinite, we can find some $\alpha \in \Omega \setminus \sigma$, and some $f \in [\sigma : \tau]$ such that $(\alpha)f \notin \Sigma_{\alpha}$. Then $f \in \transO\setminus Y \subseteq \transO\setminus X$. Since sets of the form $[\sigma : \tau]$ constitute a base for $\T$, this implies that $\transO \setminus \overline{X} = \transO \setminus X$ is dense in $\transO$, and hence $X$ is nowhere dense.
\end{proof}

\begin{prop} \label{compact-prop}
Let $\, \Omega$ be a set, and suppose that $\, \transO$ is a semitopological semigroup with respect to some $T_1$ topology. If $X \subseteq \transO$ is compact, then $X$ is closed and $|(\alpha)X| < \aleph_0$ for all $\alpha \in \Omega$.
\end{prop}

\begin{proof}
By Theorem~\ref{point-finer}, if $\T$ is a $T_1$ topology with respect to which $\transO$ is a semitopological semigroup, then $\T$ contains the pointwise topology. Hence a set compact in such a topology $\T$ is also compact in the pointwise topology. The desired conclusion now follows from Lemma~\ref{compact-lem}(1).
\end{proof}

In general, it is not the case that all subsets of $\transO$ (for $\Omega$ infinite) that are closed and satisfy $|(\alpha)X| < \aleph_0$ for all $\alpha \in \Omega$ are compact in a $T_1$ semigroup topology. For example, consider the discrete topology, where only finite sets are compact.

Recall that a topological space is \emph{locally compact} if every element has an open neighbourhood that is contained in some compact set. The remainder of this section is devoted to showing that locally compact semigroup topologies on $\transO$ are generally not well-behaved.

\begin{cor} \label{loc-com-base}
Let $\, \Omega$ be an infinite set, and suppose that $\, \transO$ is a semitopological semigroup with respect to some locally compact $T_1$ topology $\, \T$. Then any base for $\, \T$ must have strictly more than $\, |\Omega|$ elements. 
\end{cor}

\begin{proof}
Suppose that there is a base $\B$ for $\T$ such that $|\B| \leq |\Omega|$. Since $\T$ is locally compact, for every $f \in \transO$ there exists some $U_f \in \B$ and a compact set $X_f \subseteq \transO$ such that $f \in U_f \subseteq X_f$. Hence $\bigcup_{f \in \transO} X_f = \transO$. Since $|\B| \leq |\Omega|$, we can find some subset $\set{X_{\alpha}}{\alpha \in \Omega}$ of $\set{X_f}{f \in \transO}$, such that $\bigcup_{\alpha \in \Omega} X_{\alpha} = \transO$. In view of each $X_{\alpha}$ being compact, it follows from Proposition~\ref{compact-prop} that $(\alpha)X_{\alpha}$ is finite for each $\alpha \in \Omega$. But since $\Omega$ is infinite, we can find some $f \in \transO$ such that $(\alpha)f \notin (\alpha)X_{\alpha}$ for all $\alpha \in \Omega$, which contradicts $\bigcup_{\alpha \in \Omega} X_{\alpha} = \transO$. Thus any base for $\, \T$ must have $>|\Omega|$ elements.
\end{proof}

We observe that while the topology on $\transO$ constructed in Proposition~\ref{open-prod-top} satisfies the conclusion of the previous result, it is not locally compact, since Proposition~\ref{open-prod-top}(6) implies that no element of $\transO$ has an open neighbourhood that is contained in a compact set.

We are now ready to give a partial analogue of the result~\cite[Section 4]{Gaughan} of Gaughan that there is no nondiscrete locally compact Hausdorff group topology on $\symO$.

\begin{theorem} \label{loc-comp-perf}
Let $\, \Omega$ be an infinite set such that $\, |\Omega|$ has uncountable cofinality, and suppose that $\, \transO$ is a semitopological semigroup with respect to a locally compact topology $\, \T$. Then $\, \T$ is either not perfect or not $T_1$.
\end{theorem}

\begin{proof}
Seeking a contradiction, suppose that $\T$ is perfect and $T_1$, and let $f \in \transO$ be a constant function. Since $\T$ is locally compact, there must be a compact set $X$ and an open neighbourhood $U$ of $f$ such that $U \subseteq X$. Letting $\Sigma_{\alpha} = (\alpha)X$ for each $\alpha \in \Omega$, we have $U \subseteq \prod_{\alpha \in \Omega} \Sigma_{\alpha}$. Since $\T$ is perfect, by Proposition~\ref{iso-prop}(1), it must be the case that $|\set{\alpha \in \Omega}{\beta \notin \Sigma_{\alpha}}| < |\Omega|$ for each $\beta \in \Omega$. Then letting $\Gamma \subseteq \Omega$ be a countably infinite set, $$|\set{\alpha \in \Omega}{\Gamma \not\subseteq \Sigma_{\alpha}}| \leq \Big|\bigcup_{\beta \in \Gamma}\set{\alpha \in \Omega}{\beta \notin \Sigma_{\alpha}}\Big| < |\Omega|,$$ since $|\Omega|$ is assumed to have uncountable cofinality. Hence $\Gamma \subseteq \Sigma_{\alpha}$ for some $\alpha \in \Omega$, making $\Sigma_{\alpha}$ infinite. But since $\T$ is $T_1$, by Proposition~\ref{compact-prop}, $|\Sigma_{\alpha}| < \aleph_0$ for all $\alpha \in \Omega$, producing the desired contradiction.
\end{proof}

In light of Theorem~\ref{loc-comp-perf} we ask the following.

\begin{question}
Does there exist an infinite set $\, \Omega$ and a perfect locally compact $T_1$ (or Hausdorff) topology on $\, \transO$ with respect to which it is a topological semigroup?
\end{question}

\section*{Acknowledgement}

We are grateful to Slawomir Solecki for bringing to our attention a result of Montgomery~\cite[Theorem 2]{Montgomery}, that plays a crucial role in the proof of Theorem~\ref{polish-thrm}.

\medskip

\noindent Z.\ Mesyan, Department of Mathematics, University of Colorado, Colorado Springs, CO, 80918, USA 

\noindent \emph{Email:} \href{mailto:zmesyan@uccs.edu}{zmesyan@uccs.edu}

\medskip

\noindent J.\ D.\ Mitchell, Mathematical Institute, North Haugh, St Andrews, Fife, KY16 9SS, Scotland

\noindent \emph{Email:} \href{mailto:jdm3@st-and.ac.uk}{jdm3@st-and.ac.uk}

\medskip

\noindent Y.\ P\'eresse, University of Hertfordshire, Hatfield, Hertfordshire, AL10 9AB, UK

\noindent \emph{Email:} \href{mailto:y.peresse@herts.ac.uk}{y.peresse@herts.ac.uk}

\end{document}